\documentclass[review]{elsarticle}

\usepackage[letterpaper,margin=1in]{geometry}
\usepackage{graphicx}

\usepackage{amsmath,amssymb,amsfonts,amsthm}
\usepackage{mathtools}
\usepackage{bm}
\usepackage{mathrsfs}

\usepackage{algorithm}
\usepackage[end]{algpseudocode}

\usepackage{enumitem}
\usepackage{microtype}
\usepackage{booktabs}
\usepackage{tabularx}
\usepackage{multirow}
\usepackage{tikz}

\usepackage{xcolor}
\usepackage{hyperref}
\usepackage{cleveref}

\hypersetup{
  colorlinks=true,
  linkcolor=cyan,
  citecolor=teal,
  urlcolor=black,
  filecolor=black
}

\makeatletter
\patchcmd{\ps@pprintTitle}
  {\hbox to \textwidth}
  {\hbox to \dimexpr\textwidth-\fontdimen2}
  {}{}
\makeatother

\theoremstyle{plain}
\newtheorem{theorem}{Theorem}[section]

\newtheorem{lemma}[theorem]{Lemma}

\theoremstyle{definition}
\newtheorem{definition}[theorem]{Definition}

\theoremstyle{remark}
\newtheorem{remark}[theorem]{Remark}

\numberwithin{equation}{section}
\numberwithin{table}{section}
\numberwithin{figure}{section}

\newcommand{\C}{\mathbb{C}}
\newcommand{\N}{\mathbb{N}}
\newcommand{\Nzero}{\mathbb{N}_0}

\newcommand{\ee}{{\rm e}}

\newcommand{\dist}{\operatorname{dist}}
\newcommand{\spann}{\operatorname{span}}

\begin{document}

\begin{frontmatter}

\title{\textbf{Nearly Optimal Kolmogorov Widths under Holomorphic Mappings}}

\author[1]{Yuwen Li\corref{corresponding}}
\ead{liyuwen@zju.edu.cn}
\cortext[corresponding]{Corresponding author.}

\author[1]{Guozhi Zhang}
\ead{gzzh@zju.edu.cn}

\affiliation[1]{organization={School of Mathematical Sciences, Zhejiang University},
            addressline={866 Yuhangtang Road},
            city={Hangzhou},
            postcode={310058},
            state={Zhejiang},
            country={China}}

\begin{abstract}
This paper establishes essentially optimal asymptotic bounds for Kolmogorov widths under holomorphic mappings between complex Banach spaces. Given a compact parameter set whose Kolmogorov widths decay algebraically with rate $s$, we prove that the widths of its image under a holomorphic mapping decay algebraically with every rate $t<s$, thereby answering an open question raised by Cohen and DeVore. As an application, we obtain a sharp characterization of the approximability of solution manifolds associated with inf-sup stable parametrized PDEs. We also construct an explicit example showing that the arbitrarily small loss in the algebraic decay exponent is unavoidable. Finally, we provide similar characterization of asymptotic bounds for Kolmogorov widths in the exponentially decaying regime. Our analysis involves multilinear Taylor expansion in Banach spaces and a novel block dyadic expansion-truncation technique.
\end{abstract}


\end{frontmatter}

\section{Introduction}\label{sec:introduction}
Projection-based model reduction, such as the reduced basis method \cite{QuarteroniManzoniNegri2016,HesthavenPagliantiniRozza2022}, constitutes an important class of numerical techniques for simultaneously solving a family of parametrized partial differential equations (PDEs). These methods project high-fidelity PDE models onto low-dimensional approximation subspaces constructed during an offline stage. Their best achievable accuracy is governed by how well the associated parametrized solution manifold can be approximated by low-dimensional linear subspaces. It
is therefore natural to ask for the best accuracy that any linear
approximation subspace with a prescribed dimension could attain. Kolmogorov
$n$-width is the key functional-analytic concept related to this question and serves as
a theoretical benchmark for the performance of projection-based model reduction methods
\cite{Pinkus1985nWidthsIA,Lorentz1996ConstructiveA,
DeVore2017Theoretical}. It is also the fundamental concept in low-rank fast solvers \cite{BebendorfHackbusch2003,EngquistZhao2018Approximate} as well as generalized and multiscale finite element methods \cite{BabuskaLipton2011,Ma2026}.

For $n\in\mathbb N$, this width for any compact set $\mathcal K$ in a Banach space $X$ is defined by
\begin{equation}
\label{eq:width-definition}
d_n(\mathcal K)_X
:=
\inf_{\substack{V_n\subset X\\ \dim V_n\leq n}}
\sup_{a\in\mathcal K}
\inf_{v\in V_n}\|a-v\|_X.
\end{equation}
Compactness of
$\mathcal K$ implies
$
\lim_{n\to\infty}d_n(\mathcal K)_X=0.
$
We refer to
\cite{Pinkus1985nWidthsIA,Lorentz1996ConstructiveA}
for basic  properties of Kolmogorov widths. Let $Y$ be a Banach space, and let
$
u:\mathcal K\to Y
$
be the associated parameter-to-solution map of a parameter-dependent PDE. Its image
$
\mathcal M:=u(\mathcal K)\subset Y
$
is commonly referred to as the solution manifold. The quantity
$
d_n(\mathcal M)_Y
$
is the smallest uniform error attainable when all elements of
$\mathcal M$ are approximated from a single linear subspace of dimension at
most $n$. Its decay therefore describes the intrinsic linear complexity
of the underlying family of solutions.
For a wide class of parametric PDEs that satisfy the inf-sup condition, the
solution map admits a holomorphic extension to a suitable complex
neighborhood of the parameter domain. This analytic dependence is a
fundamental ingredient in sparse polynomial approximation and in the
derivation of dimension-robust convergence estimates for
high-dimensional solution families
\cite{CohenDeVoreSchwab2011,Cohen_DeVore_2015}. It also suggests that the
approximability of the solution manifold involved in the PDE structure is related to that of
the accessible parameter set.

The sequence $(d_n(\mathcal M)_Y)_{\{n\geq0\}}$ is particularly relevant to reduced basis methods that construct problem-adapted spaces
from selected solution snapshots, typically of the form
$$
V_n
=
\operatorname{span}\{u(a_1),\ldots,u(a_n)\},
$$
where $a_1,\ldots,a_n\in\mathcal K$ are chosen during an offline stage
\cite{MadayPateraTurinici2002,Maday2007,
QuarteroniManzoniNegri2016,HesthavenPagliantiniRozza2022}.
Although such snapshot spaces need not realize the infimum in the
definition of $d_n(\mathcal M)_Y$, their convergence can often be
controlled by the widths of $\mathcal M$. In particular, if $Y$ is a
Hilbert space and
$
d_n(\mathcal M)_Y=O(n^{-s}),
$ with $s>0$
then suitable reduced basis greedy algorithms converge at the same
algebraic rate
\cite{BinevEtAl2011}. Corresponding results in Banach spaces were established in
\cite{DeVorePetrovaWojtaszczyk2013,Wojtaszczyk2015} and generalized to the empirical interpolation method \cite{MadayMulaTurinici2016}, see also recent results \cite{LiSiegel2024,Li2025EIM,LiWang2026PODGreedy} using
entropy numbers to obtain refined convergence estimates for reduced basis greedy algorithms. These results show that Kolmogorov
width is not only an abstract concept in approximation theory and functional analysis, but also
a reference rate for the analysis of practical reduced basis
constructions. 

Suppose that
$\mathcal K\subset X$ admits accurate finite-dimensional linear
approximations and that
$
u:\mathcal O\to Y
$
is holomorphic on an open neighborhood $\mathcal O$ of $\mathcal K$. The preceding considerations lead to a general question that is to what extent the decay of
$
d_n(\mathcal K)_X
$ is
preserved under the nonlinear transformation $u$, 
independent of any particular PDE model.
This problem was initiated by Cohen and DeVore in
\cite{Cohen_DeVore_2015,cohen2016kolmogorov}. In particular, for $s>1$ they proved that 
\begin{equation}\label{eq:CohenDeVore}
d_n(\mathcal K)_X=O(n^{-s})\Longrightarrow d_n(u(\mathcal K))_Y=O(n^{-t}),\qquad\forall\, t<s-1.
\end{equation}
Thus, their general argument entails a loss of one
full power in the decay exponent. More precise width estimates are known for
particular PDE families. The reduced-basis construction of Maday, Patera, and
Turinici implies exponential width decay for single-parameter symmetric
coercive elliptic problems \cite{MadayPateraTurinici2002}. Melenk obtained sharp width bounds for classes
of elliptic problems that include rough coefficients and singular
perturbations \cite{Melenk2000}. For a linear transport model, exact width
formulas relate the decay rate to the regularity of the data
\cite{ArbesGreifUrban2025}. For special $\mathcal K$ and $u$ arising from
parametric elliptic PDEs, the rate comparison \eqref{eq:CohenDeVore} can be
improved to $t=s$ by exploiting additional elliptic structure of the underlying
model \cite{bachmayr2017kolmogorov}. 

The purpose of this paper is to determine the optimal asymptotic 
behavior of Kolmogorov widths under general holomorphic mappings. Our
main result shows that the loss of one full power in \eqref{eq:CohenDeVore} can be reduced to an arbitrarily small loss, see the next theorem. The proof combines localized multilinear Taylor expansions in Banach spaces with a block dyadic expansion that preserves the finite-dimensional structure of the
approximation spaces associated with $\mathcal K$, see
Section~\ref{Sec:proof_thm1}.

\begin{theorem}
\label{thm:main_result1}
Suppose that $\mathcal{O}\subset X$ is open, $\mathcal{K}\subset \mathcal{O}$ is compact, $u:\mathcal{O}\to Y$ is holomorphic, and there is a constant $C<\infty$ such that
$\sup_{a\in \mathcal{O}}\|u(a)\|_Y\le C.$
Given $s>0$, if 
\[
\sup_{n\ge1}n^s d_n(\mathcal{K})_X<\infty,
\]
then for any
 $0<t<s$, we have
\[
\sup_{n\ge1}n^t d_n(u(\mathcal{K}))_Y<\infty.
\]
\end{theorem}

The holomorphy and boundedness assumptions in Theorem~\ref{thm:main_result1}
hold for broad classes of real-valued parametrized PDEs after
complexification. Let $X_{\mathbb R}$ be a real Banach parameter space,
$\mathcal K\subset X_{\mathbb R}$ be compact, and $Y_{\mathbb R}$ be a real Hilbert space of trial and test functions.
Consider the real variational problem: find $u(a)\in Y_{\mathbb R}$ such that
\[
b_a(u(a),v)=\langle f,v\rangle,\qquad v\in Y_{\mathbb R},
\]
where $a\in\mathcal K$ is the varying parameter, $f\in Y_{\mathbb R}'$ is fixed and $b_a$ is a bounded bilinear form.
Let the following inf-sup condition hold for some $\beta>0$ independent of $a\in\mathcal K$: 
\begin{equation}\label{eq:infsup}
  \inf_{0\ne w\in Y_{\mathbb R}}\sup_{0\ne v\in Y_{\mathbb R}}
\frac{b_a(w,v)}{\|w\|_{Y_{\mathbb R}}\|v\|_{Y_{\mathbb R}}}=\inf_{0\ne v\in Y_{\mathbb R}}\sup_{0\ne w\in Y_{\mathbb R}}
\frac{b_a(w,v)}{\|w\|_{Y_{\mathbb R}}\|v\|_{Y_{\mathbb R}}}\ge\beta.
\end{equation}
The associated real operator $B(a):Y_{\mathbb R}\to Y_{\mathbb R}'$, induced by $\langle B(a)v,w\rangle=b_a(v,w)$,  is
then invertible with $\|B(a)^{-1}\|\le\beta^{-1}$.
Let $X$ and $Y$ be the complexifications of
$X_{\mathbb R}$ and $Y_{\mathbb R}$,
respectively. Suppose that the sesquilinear extension of $b_a$ to an open neighborhood
$\mathcal U\subset X$ of $\mathcal K$ is holomorphic in $a\in\mathcal U$.
By compactness of $\mathcal K$ and continuity of the induced operator family,
there is a smaller open neighborhood $\mathcal K\subset\mathcal O\subset\mathcal U$
on which the extended solution map $a\mapsto u(a)=B(a)^{-1}f$ is holomorphic and satisfies
$\sup_{a\in\mathcal O}\|u(a)\|_Y\le2\beta^{-1}\|f\|$;
see \cite[Sections~2.1--2.2]{Cohen_DeVore_2015}.
For example, on a bounded domain $\Omega$, the parametric bilinear form and parameter set  
\begin{align*}
  b_a(v,w)&=\int_\Omega(2+a^2+\sqrt{1+a^2})\nabla v\cdot\overline{\nabla w},
\qquad X=L^\infty(\Omega,\mathbb{C}),\quad Y=H_0^1(\Omega,\mathbb{C}),\\
\mathcal{K}&=\big\{a\in L^\infty(\Omega): \|a\|_{C^{0,\alpha}(\Omega)}\leq1\big\}\text{ for some }\alpha\in(0,1)
\end{align*}
satisfies \eqref{eq:infsup} and $b_a$ holomorphically depends on $a$ in a sufficiently small open neighborhood of $\mathcal{K}$.
The same mechanism applies to families of parabolic and Stokes equations with holomorphic parameter dependence. 
Theorem~\ref{thm:main_result1} also applies to nonlinear parametrized PDEs whenever their parameter-to-solution maps admit bounded holomorphic extensions, as may be established locally by the holomorphic implicit function theorem under suitable invertibility assumptions; see \cite[Section~2.3]{Cohen_DeVore_2015}.

Our second result shows that this almost-preservation theorem is sharp.
In particular, the endpoint estimate
$
d_n(u(\mathcal K))_Y=O(n^{-s})
$
cannot hold uniformly over the class of all holomorphic mappings.
Nevertheless, endpoint preservation may occur under additional assumptions \cite{bachmayr2017kolmogorov}. 

\begin{theorem}
\label{thm:main_result2}
For any $s>0$, there exist separable complex Hilbert spaces $X$ and $Y$, a compact set $\mathcal{K}\subset X$, an open set $\mathcal{O}\subset X$ containing $\mathcal{K}$, and a holomorphic map $u:\mathcal{O}\to Y$ such that:
\begin{enumerate}[label=(\roman*),leftmargin=2.4em]
\item $u$ is the restriction to $\mathcal{O}$ of a continuous two-homogeneous polynomial on $X$;
\item $u$ is uniformly bounded on $\mathcal{O}$;
\item $
d_n(\mathcal{K})_X \asymp n^{-s}$ and  $d_n(u(\mathcal{K}))_Y\asymp n^{-s}\bigl(\log(n+1)\bigr)^s$
 for $n\ge1$.
\end{enumerate}
\end{theorem}

Motivated by the almost optimal preservation of algebraic decay, it is
natural to ask whether an analogous phenomenon occurs in the exponentially convergent regime. The next two results  provide a negative answer to this
question.

\begin{theorem}\label{thm:main_result3}
Suppose that there exist $C_0,c_0,\alpha>0$ such that
\[
 d_n(\mathcal{K})_X\le C_0\ee^{-c_0n^\alpha},
 \qquad n\ge1.
\]
Then under the assumptions of Theorem \ref{thm:main_result1}, there are constants $C_1,c_1>0$ such that
$$
 d_n(u(\mathcal{K}))_Y
 \le C_1\ee^{-c_1[\log(n+2)]^{\alpha+1}},
 \qquad n\ge1.
$$
\end{theorem}

\begin{theorem}\label{thm:main_result4}
For any $\alpha>0$, there exist separable complex Hilbert spaces $X$ and $Y$, a compact set $\mathcal{K}\subset X$, and a mapping $u:X\to Y$  holomorphic in any bounded subset of $X$, such that 
\begin{align*}
  C_1\ee^{-c_1n^\alpha}
 &\le d_n(\mathcal{K})_X
 \le C_2\ee^{-c_2n^\alpha},\\
  C_3\ee^{-c_3[\log(n+2)]^{\alpha+1}}
 &\le d_n(u(\mathcal{K}))_Y
 \le C_4\ee^{-c_4[\log(n+2)]^{\alpha+1}}
\end{align*}
for any $n\geq1$ and some absolute constants $C_1,c_1,C_2,c_2,C_3,c_3,C_4,c_4>0$.
\end{theorem}

The example in
Theorem~\ref{thm:main_result4} matches the decay rate of $d_n(u(\mathcal K))_Y$ in
Theorem~\ref{thm:main_result3}. Consequently, for general holomorphic
mappings, no estimate of the form
$
d_n\bigl(u(\mathcal K)\bigr)_Y
\leq
C\ee^{-cn^\beta}
$ for some $\beta>0$,
can follow solely from  exponential decay of
$
d_n(\mathcal K)_X
$.

The rest of the paper is organized as follows.
In Section~\ref{Sec:Preliminaries}, we introduce the notation and recall
basic results on complex analysis in Banach spaces.
Section~\ref{Sec:proof_thm1} is devoted to the proof of
Theorem~\ref{thm:main_result1}, while
Section~\ref{Sec:proof_thm2} contains the construction of examples in
Theorem~\ref{thm:main_result2}.
The proofs of Theorems~\ref{thm:main_result3} and
\ref{thm:main_result4} are given in
Section~\ref{Sec:proof_thm3_4}.
Finally, concluding remarks are presented in
Section~\ref{Sec:conclusion}.

\section{Preliminaries}\label{Sec:Preliminaries}
\subsection{Notation}
\label{subsec:notation}
Throughout this paper, $X$ and $Y$ denote complex Banach spaces, and
$\|\cdot\|_X$ and $\|\cdot\|_Y$ denote their respective norms. We write
$\mathbb N_+:=\{1,2,\ldots,\}$ and $
\mathbb N:=\mathbb N_+\cup\{0\}
$. For $m\in\mathbb N_+$, we denote by $X^m$ the $m$-fold Cartesian product of $X$, equipped with the product norm. For $x\in\mathbb R$, we denote by $\lfloor x\rfloor$ the greatest
integer not exceeding $x$.
For $a\in X$ and $r>0$, let
$B_X(a,r):=\{x\in X:\|x-a\|_X<r\}.$
The distance from $x\in X$ to a subset $V\subset X$ is 
$$
\operatorname{dist}(x,V)_X
:=
\inf_{v\in V}\|x-v\|_X.
$$
For two nonnegative sequences $(a_n)_{\{n\geq1\}}$ and
$(b_n)_{\{n\geq1\}}$, we write
$
a_n=O(b_n)
$ if there exists a constant $c>0$, independent of $n$, such that
$a_n\leq cb_n$ for all $n\geq1$. We write
$
a_n\asymp b_n
$
if both $a_n=O(b_n)$ and $b_n=O(a_n)$ hold.
For a countable index set $\Lambda$, we write
$$
\ell_2(\Lambda;\mathbb C)
:=
\left\{
x=(x_\lambda)_{\{\lambda\in\Lambda\}}:
\sum_{\lambda\in\Lambda}|x_\lambda|^2<\infty,\,\text{each }x_\lambda\in\mathbb{C}
\right\}.
$$
Its standard orthonormal basis is denoted by
$(e_\lambda)_{\{\lambda\in\Lambda\}}$, unless another symbol is
specified. For a finite set $S$, we denote its cardinality by $|S|$ or $\#S$.
\subsection{Complex analysis in Banach spaces}
In this section we recall theory of holomorphic mappings between Banach spaces via power series expansions. For further details, we refer the reader to \cite[Chapter II]{mujica2010complex}. Homogeneous polynomials play a fundamental role in the definition of power series in Banach spaces. We first recall their definition as follows.

\begin{definition}[\bf Homogeneous polynomial]\label{Def:Homogeneous_polynomial}
A mapping $P:X\to Y$ is said to be an $m$-homogeneous polynomial if there exists an $m$-linear mapping
$A:X^m\to Y$ such that
\[
P(x)=Ax^m:=A(\underbrace{x,x,\ldots,x}_{m\ \text{times}})
\]
for every $x\in X$.
We shall denote by 
$\mathcal{P}^{m}(X;Y)$ the vector space of all continuous $m$-homogeneous polynomials from $X$ into $Y$. For each
$P\in\mathcal{P}^{m}(X;Y)$ we set
\[
\|P\|
=
\sup_{x\in X,\ \|x\|\leq 1}\|P(x)\|_Y.
\]
\end{definition}

The power-series definition of holomorphic mappings adopted below is equivalent to the one based on complex Fr\'echet differentiability used in \cite{cohen2016kolmogorov,Cohen_DeVore_2015}; see \cite[Theorem~13.16]{mujica2010complex} and \cite[Theorem 7]{mujica2006holomorphic}.

\begin{definition}[\bf Holomorphic mappings]\label{Def:Holomorphic_mappings}
Let $\mathcal{O}$ be an open subset of $X$. A mapping
$u: \mathcal{O}\to Y$ is said to be holomorphic if for any
$a\in \mathcal{O}$ there exist a ball $B_X(a;r)\subset \mathcal{O}$ and a sequence of
polynomials
$
P_{a,m}\in\mathcal{P}^{m}(X;Y)
$
such that
\begin{equation*}
u(x)=u(a)+\sum_{m=1}^{\infty}P_{a,m}(x-a)
\end{equation*}
uniformly for $x\in B_X(a;r)$. 
\end{definition}

\begin{lemma}[Corollary 7.4 from \cite{mujica2010complex}]\label{Le:Cauchy_inequality}
Let $\mathcal{O}$ be an open subset of $X$ and $u:\mathcal{O}\to Y$ be holomorphic as in Definition \ref{Def:Holomorphic_mappings}.
Then for any
$m\in\mathbb{N}_+$, $r>0$, $a\in\mathcal{O}$ and $z\in X$ with $\{a+\zeta z:|\zeta|\le r\}\subset\mathcal O$, we have the Cauchy inequality
\[
\left\|P_{a,m}(z)\right\|_Y
\leq
r^{-m}
\sup_{\zeta\in\mathbb{C}, |\zeta|=r}
\|u(a+\zeta z)\|_Y.
\]
\end{lemma}

\begin{lemma}[Theorem 2.2 from \cite{mujica2010complex}]
\label{Le:homogeneous_polynomial_normbound}
For any continuous and symmetric $m$-linear mapping
$A:X^m\to Y$, if we let 
$P:X\to Y$ be the $m$-homogeneous polynomial defined by $P(x)=Ax^m$.
Then it holds that
$$
\|A\|_{\mathcal{L}^m(X,Y)}:=\sup_{\|x_i\|\leq1,\, i=1, \ldots, m} \|A(x_1,\ldots,x_m)\|_Y
\leq
\frac{m^m}{m!}\|P\|
.$$
\end{lemma}

\subsection{An orthogonal-spike estimate for Kolmogorov widths}

The computation of lower bound of the two Kolmogorov $n$-widths appearing Theorems \ref{thm:main_result2} and \ref{thm:main_result4}
will finally reduce to the following elementary lemma in Hilbert spaces.
\begin{lemma}
\label{lem:orthogonal-spikes}
Let $H$ be a Hilbert space, let $(e_k)_{\{k\geq 1\}}$ be an
orthonormal basis of $H$, and let $(c_k)_{\{k\geq 1\}}$ be a positive
nonincreasing sequence converging to zero.
For $
\Sigma:=\{0\}\cup\{c_k e_k:k\geq1\}$,
one has
\begin{equation*}
\frac{c_{2n}}{\sqrt2}
\leq
d_n(\Sigma)_H
\leq
c_{n+1},
\qquad n\geq1.
\end{equation*}
\end{lemma}

\begin{proof}
For the upper bound, choose
$
V_n:=\spann\{e_1,\ldots,e_n\}.
$
Then $c_k e_k\in V_n$ for $k\leq n$, while for $k>n$,
$$
\dist(c_k e_k,V_n)_H
=
c_k
\leq
c_{n+1}.
$$
Hence
$
d_n(\Sigma)_H\leq c_{n+1}.
$ 

For the lower bound, let $V\subset H$ satisfy $\dim V\leq n$, and let
$P_V$ denote the orthogonal projection onto $V$. If
$(v_\ell)_{\{1\leq\ell\leq\dim V\}}$ is an orthonormal basis of $V$,
then Bessel's inequality yields
$$
\begin{aligned}
\sum_{k=1}^{2n}\|P_Ve_k\|_H^2
=
\sum_{\ell=1}^{\dim V}
\sum_{k=1}^{2n}
|\langle e_k,v_\ell\rangle_H|^2\leq
\dim V
\leq n.
\end{aligned}
$$
Therefore, we have 
$$
\sum_{k=1}^{2n}\dist(e_k,V)_H^2
=
2n-\sum_{k=1}^{2n}\|P_Ve_k\|_H^2
\geq n.
$$
Hence there exists $k\in\{1,\ldots,2n\}$ such that
$
\dist(e_k,V)_H\geq 2^{-1/2}.
$
Since $c_k\geq c_{2n}$,
$$
\sup_{z\in\Sigma}\dist(z,V)_H
\geq
\dist(c_k e_k,V)_H
=
c_k\dist(e_k,V)_H
\geq
\frac{c_{2n}}{\sqrt2}.
$$
Taking the infimum over all subspaces $V\subset H$ with
$\dim V\leq n$ yields the lower bound.
\end{proof}
\section{Proof of Theorem \ref{thm:main_result1}}\label{Sec:proof_thm1}
Since $\mathcal{K}\subset \mathcal{O}$ is compact, there exists $\rho>0$ sufficiently small such that
\[
\{z\in X:\operatorname{dist}(z,\mathcal{K})_X<2\rho\}\subset \mathcal{O}.
\]
Consequently, $B_X(a,2\rho)\subset \mathcal{O}$ for every $a\in\mathcal{K}$.
Throughout this section, we assume that $u$ satisfies the assumptions of Theorem \ref{thm:main_result1}.
\subsection{Multilinear Taylor expansions of holomorphic mappings}

\begin{lemma}
\label{lem:taylor-bound}
For every $a\in \mathcal{K}$, there exists continuous and symmetric $m$-linear mappings $A_{a,m}:X^m\to Y$ such that
\begin{equation*}
u(a+h)=u(a)+\sum_{m=1}^{\infty}A_{a,m}(h,\ldots,h),
\end{equation*}
for any $h\in X$ with $\|h\|_X<2\rho$. In addition, it holds that for any $a\in \mathcal{K}$,
\begin{equation}\label{eq:multilinear-cauchy}
\|A_{a,m}\|_{\mathcal{L}^m(X;Y)}
 \le C\left(\frac{\ee}{\rho}\right)^m,
 \qquad m\ge1.
\end{equation}
\end{lemma}

\begin{proof}
The first assertion follows immediately from Definition \ref{Def:Holomorphic_mappings}. The symmetry of $A_{a,m}$ follows from
\cite[Theorem~2.2]{mujica2010complex}. Let  
$P_{a,m}(h)=A_{a,m}(h,\ldots,h)$. By combining Lemma \ref{Le:Cauchy_inequality} with $r=\rho$ and using the fact that $u$ is bounded by $C$, we obtain 
\begin{equation}\label{Eq:polynomial-bound}
 \|P_{a,m}\|_Y=\sup_{\|h\|_X\leq1}\|P_{a,m}(h)\|_Y\le C\rho^{-m} 
\end{equation}
Combining Lemma \ref{Le:homogeneous_polynomial_normbound} with \eqref{Eq:polynomial-bound} and the inequality $m!\geq (m/\ee)^m$ yields \eqref{eq:multilinear-cauchy}.  
\end{proof}

\subsection{Kolmogorov \texorpdfstring{$n$}{n}-width upper bounds in an abstract framework}

The following lemma is a new block $n$-term truncation technique that replaces the scalar best $n$-term truncation used in \cite{cohen2016kolmogorov} and \cite{Cohen_DeVore_2015}. It is the key mechanism behind the improved decay exponent.

\begin{lemma}\label{Le:block-truncation_algebraic_width}
Let $0<p<1$ and let $\mathcal{F}\subset Y$. Suppose that there exist a countable index set 
$\mathcal{I}$, finite-dimensional subspaces $Y_\lambda\subset Y$ with
$
1\leq \dim Y_\lambda\le D_\lambda,
$
and numbers $a_\lambda\ge0$ for each $\lambda\in\mathcal{I}$, together with a finite-dimensional subspace
$Y_0\subset Y$ satisfying
$
\dim Y_0\le d_0,
$
such that
\begin{enumerate}[label=(\roman*),leftmargin=2.4em]
\item Every $f\in\mathcal{F}$ admits a representation
$$
f=f_0+\sum_{\lambda\in\mathcal{I}}f_\lambda,
$$
where $f_0\in Y_0$ and $f_\lambda\in Y_\lambda$, and the series converges absolutely and uniformly with respect to $f\in\mathcal{F}$;

\item 
$
\|f_\lambda\|_Y\le a_\lambda
$ for every $\lambda\in\mathcal{I}$;

\item $
S_p:=\sum_{\lambda\in\mathcal{I}}D_\lambda^{1-p}a_\lambda^p<\infty.
$
\end{enumerate}
Then for any integer $n>d_0$, we have 
$$
d_n(\mathcal{F})_Y
\le
S_p^{1/p}(n-d_0)^{-(1-p)/p}.
$$
\end{lemma}

\begin{proof}
The case $S_p=0$ is immediate, and hence we assume that $S_p>0$. Let $\delta=\left(\frac{S_p}{n-d_0}\right)^{1/p}>0$, retain the blocks satisfying $a_\lambda/D_\lambda>\delta$.
The sum of their dimensions is bounded by
\begin{equation}\label{eq:retained-dimension}
\sum_{a_\lambda/D_\lambda>\delta}D_\lambda
 \le \delta^{-p}
 \sum_{\lambda\in\mathcal{I}}D_\lambda^{1-p}a_\lambda^p
 =\delta^{-p}S_p.
\end{equation}
For the discarded blocks, we have
\begin{equation}\label{eq:discarded_error}
\sum_{a_\lambda/D_\lambda\le\delta}a_\lambda=\sum_{a_\lambda/D_\lambda\le\delta}
 D_\lambda^{1-p}a_\lambda^p
 \left(\frac{a_\lambda}{D_\lambda}\right)^{1-p}
 \le \delta^{1-p}S_p.
\end{equation}
Let $Y_\delta$ be the sum of $Y_0$ and all retained $Y_\lambda$:
\[
Y_\delta = Y_0+\sum_{a_\lambda/D_\lambda>\delta} Y_\lambda.
\]
The absolute convergence assumption makes the truncated sum legitimate, and \eqref{eq:retained-dimension} gives
\[
 \dim Y_\delta\le d_0+\delta^{-p}S_p=n.
\]
It then follows from \eqref{eq:discarded_error} that 
\[
 d_n(\mathcal{F})_Y
\le\sup_{f\in\mathcal{F}}\dist(f,Y_\delta)_Y
 \le \delta^{1-p}S_p
 =S_p^{1/p}(n-d_0)^{-(1-p)/p}.
\]
The proof is complete.
\end{proof}

\subsection{Proof of Theorem \ref{thm:main_result1}}

Fix $0<t<s$
 and let $
 p:=1/(1+t)\in(0,1)$.
Then we have
\[
\frac{1-p}{p}=t,\qquad
 \gamma:=sp-(1-p)=p(s-t)>0.
\]

\noindent\textbf{Step 1: Dyadic approximation.} Let $C_A:=1+\sup_{n\ge1}n^s d_n(\mathcal{K})_X<\infty$.
By the definition of Kolmogorov $n$-width, for any $k\in \mathbb N$, there exists a subspace $\widetilde V_k\subset X$ with $ \dim\widetilde V_k\le2^k$ such that
$$
\sup_{x\in \mathcal{K}}\inf_{v\in \widetilde V_k}\|x-v\|_X
\le C_A2^{-sk}.
$$
Define $V_k:=\widetilde V_0+\cdots+\widetilde V_k$ for all $k\in \mathbb N$. Then
$
V_0\subset V_1\subset V_2\subset\cdots\subset X,
$
and $\widetilde V_k\subset V_k$. Since enlarging the approximation space does not increase the distance, we obtain 
$$
\sup_{x\in\mathcal{K}}\inf_{v\in V_k}\|x-v\|_X
\le C_A2^{-sk}.
$$
Moreover, $\dim V_k\le2^{k+1}.
$ For each $x\in\mathcal{K}$ we choose a best approximant
$\pi_k(x)\in V_k$ such that
\begin{equation}\label{Eq:bestapproximation_pi}
\|x-\pi_k(x)\|_X
=
\operatorname{dist}(x,V_k)_X
\leq C_A2^{-sk}.
\end{equation}

\vspace{2mm}
\noindent\textbf{Step 2: Localized decomposition into blocks.} For any $J\in\mathbb{N}$ and any $a,x\in\mathcal{K}$ satisfying
\begin{equation}\label{eq:local-close}
\|x-a\|_X\leq 2^{-sJ},
\end{equation}
we write 
\begin{equation}\label{eq:gr}
g_r(x,a):=\begin{cases}
\pi_J(x)-\pi_J(a),&\quad r=0,\\    
\bigl(\pi_{J+r}(x)-\pi_{J+r-1}(x)\bigr)-\bigl(\pi_{J+r}(a)-\pi_{J+r-1}(a)\bigr),&\quad r\geq 1.
\end{cases}
\end{equation}
By the nestedness of the spaces $V_k$, we have $g_r(x,a)\in V_{J+r}$. Since best approximants need not be unique, for each fixed pair $a,x\in\mathcal{K}$ we choose the required best approximants and define $g_r(x,a)$ accordingly. Thus, once these choices have been made and $J$ is fixed, $g_r(x,a)$ denotes a fixed element of $V_{J+r}$. It follows from \eqref{Eq:bestapproximation_pi} that the corresponding telescoping series converges in $X$ and
yields the representation
\begin{equation*}
x-a=\sum_{r=0}^{\infty}g_r(x,a)
\qquad\text{in }X.
\end{equation*}
Combining \eqref{Eq:bestapproximation_pi} and \eqref{eq:local-close}
with the definition \eqref{eq:gr}, we obtain
\begin{equation*}
\|g_r(x,a)\|_X
\leq C_G 2^{-s(J+r)},
\qquad \text{for all }r\in\mathbb{N},
\end{equation*}
where $
C_G:=\max\bigl\{1+2C_A,\;2C_A(1+2^s)\bigr\}$.
Consequently, the series is absolutely convergent in $X$.

\vspace{2mm}
\noindent\textbf{Step 3: Upper bounds of Kolmogorov $n$-width.} For all $r\in \mathbb N$, set
\begin{align*}
  \beta_r&:=\frac{\ee C_G}{\rho}\cdot 2^{-s(J+r)},\qquad
 d_r:=2^{J+r+1},\\
 \Theta_J&:=\sum_{r=0}^{\infty}\beta_r=\frac{\ee C_G 2^{-sJ}}{\rho(1-2^{-s})},\qquad Q_J:=\sum_{r=0}^{\infty}d_r^{1-p}\beta_r^p=\frac{2^{1-p}(\ee C_G/\rho)^p 2^{-\gamma J}}{1-2^{-\gamma}}.
\end{align*}
Since $\gamma=p(s-t)>0$, we choose $J$ sufficiently large so that
$2^{-sJ}<2\rho$, $\Theta_J<1$, and $Q_J<1$.
By the compactness of $\mathcal{K}$, there exist finitely many points
$a_1,\ldots,a_M\in\mathcal{K}$ such that
\begin{equation}\label{eq:finite-cover}
\mathcal{K}
\subset
\bigcup_{i=1}^M B_X\bigl(a_i,2^{-sJ}\bigr)\subset \mathcal{O}.
\end{equation}
For each $x\in\mathcal{K}$, choose an index $i=i(x)\in\{1,\ldots,M\}$ such that
$x\in B_X\bigl(a_i,2^{-sJ}\bigr)$.
Applying the construction in \textbf{Step 2} with $a=a_i$ and setting
$g_{i,r}(x):=g_r(x,a_i)$, we obtain
\begin{equation}\label{eq:x_ai_patch-blocks}
x-a_i=\sum_{r=0}^{\infty}g_{i,r}(x),
\end{equation}
where $
g_{i,r}(x)\in V_{J+r}$ and 
\[
\|g_{i,r}(x)\|_X
\leq C_G2^{-s(J+r)}.
\]
By the definition of $\Theta_J$, we have
\begin{equation*}
\frac{\ee}{\rho}\sum_{r=0}^{\infty}\|g_{i,r}(x)\|_X
\leq \Theta_J<1.
\end{equation*}

Let $A_{i,m}:=A_{a_i,m}$ be the multilinear Taylor coefficient from Lemma~\ref{lem:taylor-bound}.  For $m\ge1$ and an ordered multi-index
$
 \mathbf r=(r_1,\ldots,r_m)\in\mathbb N^m,
$
and $x\in\mathcal{K}$, let $i=i(x)$ be the index selected above and define
\begin{equation*}
 F_{i,m,\mathbf r}(x)
 :=A_{i,m}\bigl(g_{i,r_1}(x),\ldots,g_{i,r_m}(x)\bigr).
\end{equation*}
By multilinearity, the decomposition \eqref{eq:x_ai_patch-blocks}, and the
multilinear Taylor estimate, we obtain the expansion
\begin{equation}\label{eq:full-block-expansion}
u(x)
=
u(a_i)
+
\sum_{m=1}^{\infty}
\sum_{\mathbf r\in\mathbb{N}^m}
F_{i,m,\mathbf r}(x),
\end{equation}
which converges absolutely in $Y$, uniformly with respect to
$x\in\mathcal{K}$. Indeed, by \eqref{eq:multilinear-cauchy}, we have
\begin{equation*}
 \sum_{m=1}^{\infty}
 \sum_{\mathbf r\in\Nzero^m}
 \|F_{i,m,\mathbf r}(x)\|_Y
\le
 C\sum_{m=1}^{\infty}
 \left(
 \frac{\ee}{\rho}
 \sum_{r=0}^{\infty}\|g_{i,r}(x)\|_X
 \right)^m
     \le C\sum_{m=1}^{\infty}\Theta_J^m<\infty\quad \text{for any }x\in \mathcal{K}.
\end{equation*}

For any $j\in \{1,2,\ldots,M\}$, $m\geq 1$, and $
 \mathbf r=(r_1,\ldots,r_m)\in\mathbb N^m
$, define 
\begin{align*}
 W_{j,m,\mathbf r}
 &:=\spann\Bigl\{
 A_{j,m}(v_1,\ldots,v_m):
 v_k\in V_{J+r_k},\, 1\leq k\leq m
 \Bigr\}\subset Y.
\end{align*}
The multilinearity of each $A_{j,m}$ shows that
\begin{equation} \label{eq:block-dimension}
 \dim W_{j,m,\mathbf r}
 \le\prod_{k=1}^m\dim V_{J+r_k}
 \le D_{m,\mathbf r}:=
\prod_{j=1}^m d_{r_j}.
\end{equation}
Moreover, \eqref{eq:multilinear-cauchy} and \eqref{eq:x_ai_patch-blocks} give
$$
\begin{aligned}
\|F_{i,m,\mathbf r}(x)\|_Y
&=\|A_{i,m}\bigl(g_{i,r_1}(x),\ldots,g_{i,r_m}(x)\bigr)\|_Y\\
 &\leq \|A_{i,m}\|_Y\prod_{j=1}^m \|g_{i,r_j}(x)\|_X\le a_{m,\mathbf r}:=
C\prod_{j=1}^m\beta_{r_j}.
 \label{eq:block-amplitude}
\end{aligned}
$$
We now verify the weighted summability condition in
Lemma~\ref{Le:block-truncation_algebraic_width}. Using \eqref{eq:block-dimension} and \eqref{eq:block-amplitude}, we obtain
$$
\begin{aligned}
 S_p
 &:=\sum_{i=1}^M
 \sum_{m=1}^{\infty}
 \sum_{\mathbf r\in\mathbb N^m}
 D_{m,\mathbf r}^{1-p}a_{m,\mathbf r}^p
= MC^p
 \sum_{m=1}^{\infty}
 \sum_{r_1,\ldots,r_m\ge0}
 \prod_{j=1}^m\bigl(d_{r_j}^{1-p}\beta_{r_j}^p\bigr)
 \\
 &=MC^p\sum_{m=1}^{\infty}Q_J^m
 =MC^p\frac{Q_J}{1-Q_J}
 <\infty.
\end{aligned}
$$
Apply Lemma~\ref{Le:block-truncation_algebraic_width} to $\mathcal{F}=u(\mathcal{K})$, with
\[
 Y_0:=\spann\{u(a_1),\ldots,u(a_M)\},
 \qquad d_0:=M,
\]
and with blocks indexed by $(j,m,\mathbf r)$.  For each fixed $x\in\mathcal{K}$, only the blocks associated with the selected
patch $i(x)$ appear in the expansion \eqref{eq:full-block-expansion}. This is
admissible in Lemma~\ref{Le:block-truncation_algebraic_width}, since all blocks corresponding
to the remaining patches may simply be taken to be zero.  Applying Lemma~\ref{Le:block-truncation_algebraic_width}, we obtain, for every $n>M$,
\begin{equation}\label{eq:rate-n-minus-M}
d_n\bigl(u(\mathcal{K})\bigr)_Y
\leq
S_p^{1/p}(n-M)^{-(1-p)/p}
=
S_p^{1/p}(n-M)^{-t}.
\end{equation}
If $n\geq 2M$, then $n-M\geq n/2$, and hence
$$
d_n\bigl(u(\mathcal{K})\bigr)_Y
\leq
2^tS_p^{1/p}n^{-t}.
$$
For the finitely many indices $1\leq n<2M$, we use
$d_n(u(\mathcal{K}))_Y\leq C$ and enlarge the constant accordingly.
Therefore, there exists a constant $C_t>0$ such that
$$
d_n\bigl(u(\mathcal{K})\bigr)_Y
\leq C_tn^{-t},
\qquad n\geq1.
$$
This completes the proof of Theorem~\ref{thm:main_result1}.
\begin{remark}

The improvement in the decay exponent can be seen from comparing the scalar and block summability conditions for elements in $\mathcal{K}$.
In the scalar coordinate framework \cite{cohen2016kolmogorov}, a subspace of dimension $D_k \asymp 2^k$ at dyadic level $k$ is treated as $2^k$ separate directions, each of magnitude $2^{-sk}$. Scalar $\ell_p$ summability requires
$\sum_{k \ge 0} 2^k (2^{-sk})^p$ is finite, that is $sp > 1$,
which yields the restricted best $n$-term approximation rate $\frac{1}{p} - 1 < s - 1$.
In contrast, our block-wise argument preserves the intrinsic geometry of each finite-dimensional approximation space. A block with dimension $D_k \asymp 2^k$ and amplitude $a_k \asymp 2^{-sk}$ contributes
\[
D_k^{1-p} a_k^p \asymp 2^{k(1-p)} 2^{-skp} = 2^{-k[sp - (1-p)]}.
\]
Summability now requires only $sp > 1 - p$, which is equivalent to
$\frac{1-p}{p} < s$.
Consequently, every decay exponent $t < s$ is admissible, eliminating the rate loss caused by replacing finite-dimensional spaces with loose rectangular coefficient boxes.
    
\end{remark}

\section{Proof of Theorem \ref{thm:main_result2}}\label{Sec:proof_thm2}

We need the following lemma, whose assumption is reasonable since, for each $\delta>0$, only finitely many pairs satisfy $(ij)^{-s}\geq\delta$. The following result is a corollary of the
classical rearrangement estimate for tensor-product sequences; see
\cite[Theorem~2]{Krieg2018}. A proof is included  in Section \ref{Sec:rearrangement} for completeness.
\begin{lemma}[Rearrangement of the products]
\label{Le:rearrangement}
Let $(b_m)_{m\geq 1}$ be the nonincreasing rearrangement, counted with
multiplicity, of the family $\bigl((ij)^{-s}\bigr)_{1\leq i<j}$. Then there
exist constants $0<c_s<C_s<\infty$ such that
\begin{equation}\label{eq:b-asymptotic}
c_s\left(\frac{\log(m+1)}{m}\right)^s
\leq b_m
\leq C_s\left(\frac{\log(m+1)}{m}\right)^s,
\qquad m\geq 1.
\end{equation}
\end{lemma}

\subsection{Proof of Theorem \ref{thm:main_result2}}

Fix $s>0$ and let $X:=\ell_2(\N_+;\mathbb C)$ with its standard
orthonormal basis $(e_i)_{\{i\geq1\}}$. Define
\begin{equation}\label{eq:K-definition}
\mathcal K
:=
\{0\}
\cup\{a_i e_i:i\geq1\}
\cup\{a_i e_i+a_j e_j:1\leq i<j\},
\qquad
a_i:=i^{-s},
\end{equation}
and let
$
\mathcal O:=\{x\in X:\|x\|_X<2\}.
$ Since each element of $\mathcal K$ has norm at most
$$
\bigl(a_1^2+a_2^2\bigr)^{1/2}
=
\sqrt{1+2^{-2s}}
<
\sqrt2<2,
$$
thus $\mathcal K\subset\mathcal O$. 
Set
$
Y:=\ell_2(\mathcal{I};\mathbb C)$,
$\mathcal{I}:=\{(i,j)\in\N_+^2:i<j\},
$
and denote the standard orthonormal basis of $Y$ by
$(f_{ij})_{\{i<j\}}$. Define the symmetric  bilinear form
$B: X\times X\to Y$ by
\begin{equation*}
B(x,z)_{ij}
:=
\frac{x_i z_j+x_j z_i}{2},
\qquad i<j,
\end{equation*}
and set the holomorphic mapping 
\begin{equation}\label{eq:U-definition}
u(x):=B(x,x)=(x_i x_j)_{\{i<j\}}.
\end{equation}
By Definition \ref{Def:Holomorphic_mappings}, $u$ is holomorphic in $X$. 

\medskip
\noindent\textbf{Step 1: Compactness of $\mathcal K$.}
For one-coordinate points $a_{i_m}e_{i_m}$ in $\mathcal
K$, either
$(i_m)_{\{m\geq1\}}$ has a bounded subsequence, and hence a constant
subsequence, or $i_m\to\infty$, in which case
$$
\|a_{i_m}e_{i_m}\|_X=a_{i_m}\longrightarrow0.
$$
For two-coordinate points
$
a_{i_m}e_{i_m}+a_{j_m}e_{j_m}
$ in $\mathcal{K}$
with $i_m<j_m$, if $(i_m)_{\{m\geq1\}}$ is unbounded, then after
passing to a subsequence $i_m\to\infty$, and hence also
$j_m\to\infty$. Therefore,
$$
\|a_{i_m}e_{i_m}+a_{j_m}e_{j_m}\|_X
\leq
\sqrt2\,a_{i_m}
\longrightarrow0.
$$
If $(i_m)_{\{m\geq1\}}$ is bounded, we may assume that $i_m=i$ is
fixed. Then either $(j_m)_{\{m\geq1\}}$ has a constant subsequence or
$j_m\to\infty$, in which case
$$
a_i e_i+a_{j_m}e_{j_m}\longrightarrow a_i e_i.
$$
All possible limits belong to $\mathcal K$, so $\mathcal K$ is compact.

\medskip
\noindent\textbf{Step 2: Kolmogorov widths of $\mathcal K$.}
We prove that
$$
d_n(\mathcal K)_X\asymp n^{-s},
\qquad n\geq1.
$$
For the upper bound, let
$
E_n:=\operatorname{span}\{e_1,\ldots,e_n\}.
$
Every one-coordinate point has distance from $E_n$ at most
$a_{n+1}$. For $a_i e_i+a_j e_j$ with $i<j$, the distance is zero if
$i,j\leq n$, at most $a_{n+1}$ if exactly one index exceeds $n$, and,
if both indices exceed $n$,
$$
\operatorname{dist}(a_i e_i+a_j e_j,E_n)_X
\leq
\bigl(a_{n+1}^2+a_{n+2}^2\bigr)^{1/2}
\leq
\sqrt2\,(n+1)^{-s}.
$$
Hence
$$
d_n(\mathcal K)_X
\leq
\sup_{x\in\mathcal K}
\operatorname{dist}(x,E_n)_X
\leq
\sqrt2\,(n+1)^{-s}.
$$
For the lower bound, let
$
\Sigma_X:=\{0\}\cup\{a_i e_i:i\geq1\}\subset\mathcal K
$.
By monotonicity and Lemma~\ref{lem:orthogonal-spikes},
$$
d_n(\mathcal K)_X
\geq
d_n(\Sigma_X)_X
\geq
\frac{a_{2n}}{\sqrt2}
=
\frac{(2n)^{-s}}{\sqrt2}.
$$

\medskip
\noindent\textbf{Step 3: Holomorphy and boundedness of $u$.} The map $b$ is well defined. Indeed, for any $x,z\in X$, using
$
|\alpha+\beta|^2\leq2(|\alpha|^2+|\beta|^2)
$,
we obtain
$$
\begin{aligned}
\|B(x,z)\|_Y^2
=
\frac14\sum_{i<j}|x_i z_j+x_j z_i|^2\leq
\frac12\sum_{i\neq j}|x_i|^2|z_j|^2\leq
\frac12\|x\|_X^2\|z\|_X^2<\infty.
\end{aligned}
$$
Hence $B(x,z)\in Y$ and
$
\|B(x,z)\|_Y\leq 2^{-1/2}\|x\|_X\|z\|_X.
$ Moreover, by the definition of $u$ in \eqref{eq:U-definition} and the boundness of $B$, we have
\begin{equation}\label{eq:U-bound}
\|u(x)\|_Y
\leq
\frac1{\sqrt2}\|x\|_X^2.
\end{equation}
Then combining the definition of $\mathcal O$ with \eqref{eq:U-bound} yields
$u$ is uniformly bounded on $\mathcal{O}$ by $2\sqrt{2}$.

\medskip
\noindent\textbf{Step 4: Structure of $u(\mathcal K)$.}
Since
$
u(a_i e_i)=0
$
and
$$
u(a_i e_i+a_j e_j)
=
a_i a_j f_{ij}
=
(ij)^{-s}f_{ij},
\qquad 1\leq i<j,
$$
we have
\begin{equation*}
u(\mathcal K)
=
\{0\}
\cup
\{(ij)^{-s}f_{ij}:1\leq i<j\}.
\end{equation*}
Let $(b_m)_{\{m\geq1\}}$ be the nonincreasing rearrangement, counted
with multiplicity, of
$
\bigl((ij)^{-s}\bigr)_{\{1\leq i<j\}}.
$
Choose an enumeration
$
\mathcal{I}=\{\iota_1,\iota_2,\ldots\}
$
such that the coefficient of $f_{\iota_m}$ is $b_m$, and set
$
g_m:=f_{\iota_m}
$.
Then $(g_m)_{\{m\geq1\}}$ is an orthonormal basis of $Y$, and
\begin{equation}\label{eq:image-reordered}
u(\mathcal K)
=
\{0\}\cup\{b_m g_m:m\geq1\}.
\end{equation}

\medskip
\noindent\textbf{Step 5: Kolmogorov widths of $u(\mathcal K)$.}
Applying Lemma~\ref{lem:orthogonal-spikes} to
\eqref{eq:image-reordered}, we obtain
\begin{equation}\label{eq:image-spike-bound}
\frac{b_{2n}}{\sqrt2}
\leq
d_n(u(\mathcal K))_Y
\leq
b_{n+1}.
\end{equation}
By Lemma~\ref{Le:rearrangement},
\begin{align*}
b_{n+1}
&\leq
C_s
\left(\frac{\log(n+2)}{n+1}\right)^s
\leq
C_{s,1}n^{-s}\bigl(\log(n+1)\bigr)^s,\\
b_{2n}
&\geq
c_s
\left(\frac{\log(2n+1)}{2n}\right)^s
\geq
c_{s,1}n^{-s}\bigl(\log(n+1)\bigr)^s.
\end{align*}
Combining these estimates with \eqref{eq:image-spike-bound}, we obtain
constants $0<c'_s<C'_s<\infty$ such that
\begin{equation*}
c'_s n^{-s}\bigl(\log(n+1)\bigr)^s
\leq
d_n(u(\mathcal K))_Y
\leq
C'_s n^{-s}\bigl(\log(n+1)\bigr)^s,
\qquad n\geq1.
\end{equation*}
Consequently, we obtain the desired estimates
$$
d_n(\mathcal K)_X\asymp n^{-s},
\qquad
d_n(u(\mathcal K))_Y
\asymp
n^{-s}\bigl(\log(n+1)\bigr)^s.
$$

\subsection{Proof of Lemma \ref{Le:rearrangement}}\label{Sec:rearrangement}

To prove Lemma \ref{Le:rearrangement}, we first recall the following theorem from \cite{Krieg2018}.

\begin{lemma}[Theorem 2 from \cite{Krieg2018} ]
\label{Le:counting}
Let $\sigma: \mathbb{N}_+\rightarrow\mathbb{R}$ be a nonincreasing sequence. For any $d\in \mathbb N_+$, its $d$-th tensor power is the sequence
$\sigma_d:\mathbb{N}_+^d\to\mathbb{R}$, where
$$
\sigma_d(n_1,\ldots,n_d)
=
\prod_{j=1}^{d}\sigma(n_j).
$$
Let $\tau$ be the
nonincreasing rearrangement of the sequence
$\sigma_d$. For any
$s>0$, the following hold:
$$
\sigma(n)\asymp n^{-s}\Longrightarrow \tau(n)
\asymp
\frac{1}{((d-1)!)^s}
n^{-s}(\log (n+1))^{s(d-1)}.
$$
\end{lemma}

\begin{proof}[Proof of Lemma \ref{Le:rearrangement}]
Let $(a_m)_{\{m\geq 1\}}$ be the nonincreasing rearrangement, counted with
multiplicity, of the family $\bigl((ij)^{-s}\bigr)_{\{i,j\geq 1\}}$.
This family is precisely the second tensor power of the sequence
$$
\sigma(k)=k^{-s},\qquad k\geq 1.
$$
Applying Lemma \ref{Le:counting} with $d=2$, we find constants
$0<A_s<B_s<\infty$ such that
\begin{equation}\label{eq:am_bound}
A_s\left(\frac{\log(m+1)}{m}\right)^s
\leq a_m
\leq
B_s\left(\frac{\log(m+1)}{m}\right)^s,
\qquad m\geq 1.
\end{equation}
Since
$
\bigl((ij)^{-s}\bigr)_{\{1\leq i<j\}}
$
is a subfamily of
$
\bigl((ij)^{-s}\bigr)_{\{i,j\geq 1\}},
$
the corresponding nonincreasing rearrangements satisfy
$
b_m\leq a_m
$
for every $m\geq 1$. Consequently, \eqref{eq:am_bound} yields
$$
b_m
\leq
B_s\left(\frac{\log(m+1)}{m}\right)^s,
\qquad m\geq 1.
$$
Observe that, as multisets,
$$
\bigl((ij)^{-s}\bigr)_{\{i,j\geq 1\}}
=
\bigl((ij)^{-s}\bigr)_{\{i<j\}}
\cup
\bigl((ij)^{-s}\bigr)_{\{i>j\}}
\cup
\bigl(i^{-2s}\bigr)_{\{i\geq 1\}}.
$$
The two off-diagonal families have the same nonincreasing
rearrangement $\bigl(b_m\bigr)_{\{m\geq 1\}}$, while the diagonal
family is already nonincreasing and has rearrangement
$\bigl(d_m\bigr)_{\{m\geq 1\}}$, where
$
d_m=m^{-2s}.
$
Consequently, $\bigl(a_m\bigr)_{\{m\geq 1\}}$ is the nonincreasing
rearrangement of the multiset union of two copies of
$\bigl(b_m\bigr)_{\{m\geq 1\}}$ and one copy of
$\bigl(d_m\bigr)_{\{m\geq 1\}}$.

For any $m\geq 1$, each copy of
$\bigl(b_k\bigr)_{\{k\geq 1\}}$ contains at most $m-1$ terms strictly
larger than $b_m$, while
$\bigl(d_k\bigr)_{\{k\geq 1\}}$ contains at most $m-1$ terms strictly
larger than $d_m$. Hence their multiset union contains at most
$3m-3$ terms strictly larger than $\max\{b_m,d_m\}$. It follows that
\begin{equation}\label{Eq:a3m_upperbound}
a_{3m}\leq \max\{b_m,d_m\},
\qquad m\geq 1.
\end{equation}
By \eqref{eq:am_bound}, we have
$$
\frac{a_{3m}}{d_m}
\geq
3^{-s}A_s m^s\bigl(\log(3m+1)\bigr)^s,
\qquad m\geq 1.
$$
Since the right-hand side tends to infinity as $m\to\infty$, there
exists an integer $m_s\geq 1$ such that
$
a_{3m}>d_m
$
for every $m\geq m_s$. Combining this inequality with
\eqref{Eq:a3m_upperbound}, we obtain
$$
a_{3m}\leq b_m,
\qquad m\geq m_s.
$$
Applying the lower bound in \eqref{eq:am_bound} once again yields
$$
\begin{aligned}
b_m
\geq
\frac{A_s}{3^s}
\left(\frac{\log(m+1)}{m}\right)^s,
\qquad m\geq m_s.
\end{aligned}
$$
To extend this estimate to all $m\geq 1$, define
$$
c_s
:=
\min\left\{
\frac{A_s}{3^s},
\min_{1\leq m<m_s}
b_m\left(\frac{m}{\log(m+1)}\right)^s
\right\},
$$
where the second minimum is omitted when $m_s=1$. Since $b_m>0$ for
every $m\geq 1$, we have $c_s>0$. It follows that
$$
b_m
\geq
c_s\left(\frac{\log(m+1)}{m}\right)^s,
\qquad m\geq 1.
$$
Together with the upper bound established above, this proves
\eqref{eq:b-asymptotic} and completes the proof.
\end{proof}

\section{Proof of Theorems \ref{thm:main_result3} and \ref{thm:main_result4}}\label{Sec:proof_thm3_4}
The following lemma is an immediate consequence of Lemma \ref{Le:block-truncation_algebraic_width}.
\begin{lemma}[Exponential block truncation]\label{Le:block-truncation_exponential_width}
Let $\mathcal F\subset Y$ satisfy the assumptions of
Lemma~\ref{Le:block-truncation_algebraic_width}. Suppose that there exist constants
$C_1,C_2,\alpha>0$ and positive family
$(E_\lambda)_{\{\lambda\in\mathcal I\}}$ such that
\begin{equation}\label{eq:partition-abstract}
\sum_{\lambda\in\mathcal I}
D_\lambda\ee^{-\tau E_\lambda}
\leq
\ee^{C_1\tau^{-1/\alpha}}\quad \text{for any }\tau\in(0,1],
\end{equation}
and
$a_\lambda\leq C_2\ee^{-E_\lambda}
$ for any $\lambda\in\mathcal I$. Then there exist constants $C_3,c_3>0$
such that
$$
d_n(\mathcal F)_Y
\leq
C_3\ee^{
-c_3\bigl(\log(n-d_0+1)\bigr)^{\alpha+1}
},
\quad n>d_0.
$$
\end{lemma}

\begin{proof}
Let $\kappa>0$ be chosen sufficiently small so that
$
C_1\kappa^{1/\alpha}\leq \frac14
$.
For sufficiently large $N$ satisfying $\kappa (\log N)^\alpha>1$, let
$
p:=\kappa^{-1}(\log N)^{-\alpha}\in(0,1)
$. Since $D_\lambda\geq1$, it follows that
$D_\lambda^{1-p}\leq D_\lambda$. Hence,
$$
\begin{aligned}
S_p
:=
\sum_{\lambda\in\mathcal I}
D_\lambda^{1-p}a_\lambda^p
\leq
C_2^p\sum_{\lambda\in\mathcal I}
D_\lambda \ee^{-pE_\lambda}.
\end{aligned}
$$
Applying \eqref{eq:partition-abstract} with $\tau=p$, we obtain
$
S_p
\leq
C_2^p\exp\left(C_1p^{-1/\alpha}\right).
$
Consequently, Lemma~\ref{Le:block-truncation_algebraic_width} with $N:=n-d_0$ gives
\begin{equation*}
d_n(\mathcal F)_Y
\leq
C_2\exp\left(C_1p^{-1-1/\alpha}\right)
N^{-(1-p)/p}.
\end{equation*}
A direct calculation gives
$
\exp\left(C_1p^{-1-1/\alpha}\right)
N^{-(1-p)/p}
\leq
\exp\left(-\frac{\kappa}{2}(\log N)^{\alpha+1}\right).
$
Hence,
$$
d_n(\mathcal F)_Y
\leq
C_2\exp\left(-\frac{\kappa}{2}(\log N)^{\alpha+1}\right)
$$
for all sufficiently large $N$. By increasing the multiplicative constant if necessary, we can absorb the finitely many remaining values of $n$, which completes the proof of the lemma.
\end{proof}

\subsection{Proof of Theorem \ref{thm:main_result3}}
The proof follows the same general strategy as that of
Theorem~\ref{thm:main_result1}, with
Lemma~\ref{Le:block-truncation_exponential_width} replacing the
corresponding algebraic block-truncation result. In addition, the
symmetry of the Taylor coefficients plays an essential role in controlling
the dimension growth of high-order blocks.

\vspace{2mm}
\noindent\textbf{Step 1: Dyadic approximation.}
By the assumption and the same argument used in the proof of
Theorem~\ref{thm:main_result1}, for every $k\in\mathbb N$ there exists
a finite-dimensional subspace $V_k\subset X$ such that
$
\dim V_k\leq 2^{k+1},
$
and the family $(V_k)_{\{k\in\mathbb N\}}$ is nested. Moreover, since
finite-dimensional subspaces of normed spaces are proximinal, for every
$x\in\mathcal K$ there exists a best approximant
$\pi_k(x)\in V_k$ satisfying
\begin{equation}
\label{eq:exponential-best-approximation}
\|x-\pi_k(x)\|_X
=
\operatorname{dist}(x,V_k)_X
\leq
C_A\ee^{-c_0 2^{\alpha k}},
\qquad
C_A:=C_0+1.
\end{equation}

\medskip
\noindent\textbf{Step 2: Localized decomposition into blocks.}
We proceed as in \textbf{Step 2} of the proof of
Theorem~\ref{thm:main_result1}. Given any $J\in\mathbb{N}$, for any $a,x\in\mathcal{K}$ satisfying
\begin{equation}
\label{eq:exponential-local-close}
\|x-a\|_X\leq \ee^{-c_02^{\alpha J}}.
\end{equation}
As in the algebraic case, we obtain the representation
$$
x-a=\sum_{r=0}^{\infty}g_r(x,a),
\qquad
g_r(x,a)\in V_{J+r},
$$
with convergence in $X$. Moreover, there exist constants
$C_G,c_G>0$, independent of $J$, $a$, and $x$, such that
\begin{equation}
\label{eq:exponential-block-size}
\|g_r(x,a)\|_X
\leq
C_G\ee^{-c_G2^{\alpha(J+r)}},
\qquad r\in\mathbb N.
\end{equation}

\medskip
\noindent\textbf{Step 3: Global Taylor blocks.}
For all $r\in \mathbb N$, set
\begin{equation}
\label{eq:exponential-beta-d}
\beta_r
:=
\frac{\ee C_G}{\rho}
\exp\left(-\frac{c_G}{2}2^{\alpha(J+r)}\right)\quad \text{and}\quad
 d_r:=2^{J+r+1}.
\end{equation}
Let $q_J:=\sum_{r=0}^\infty\beta_r$.
Since both $\ee^{-c_0 2^{\alpha J}}$ and $q_J$ tend to zero as
$J\to\infty$, we may choose $J$ sufficiently large so that
$
\ee^{-c_0 2^{\alpha J}}<2\rho
$
and
$
q_J<1,
$
where $\rho>0$ is the constant fixed in the proof of
Theorem~\ref{thm:main_result1}, depending only on $\mathcal K$ and
$\mathcal O$. Since $J$ is fixed throughout the remainder of the proof, we suppress
the subscript $J$ and write $q$ in
place of $q_J$.
By compactness of $\mathcal{K}$, there exist $a_1,\ldots,a_M\in\mathcal K$ such that
\begin{equation}
\label{eq:exponential-finite-cover}
\mathcal K
\subset
\bigcup_{i=1}^M
B_X\left(a_i,\ee^{-c_02^{\alpha J}}\right)
\subset\mathcal O.
\end{equation}
For each $x\in\mathcal K$, select one index $i=i(x)$ for which $x$ belongs to the corresponding ball, and let
$
g_{i,r}(x):=g_r(x,a_i)
$  obtained from $\textbf{Step 2}$ for all $r\in \mathbb N $.
Then we have
\begin{equation}\label{eq:x_ai_patch-blocks_exponential}
x-a_i=\sum_{r=0}^{\infty}g_{i,r}(x),
\end{equation}
The choice of $J$ ensure that the Taylor expansion around $a_i$ converges absolutely.

For the remainder of this section, let $\mathfrak N$ denote the set of all finitely supported sequences
$
\nu=(\nu_j)_{\{j\geq0\}}
$
of nonnegative integers. 
For any $\nu\in\mathfrak N$ we write
$$
\nu!:=\prod_{r\geq0}\nu_r!,\qquad|\nu|:=\sum_{r\geq0}\nu_r.$$
For $m\ge1$ and an ordered multi-index
$
 \mathbf r=(r_1,\ldots,r_m)\in\mathbb N^m,
$
and $x\in\mathcal{K}$ with $i=i(x)$ be the index selected above and define
$F_{i,m,\mathbf r}(x):=A_{i,m}\bigl(g_{i,r_1}(x),\ldots,g_{i,r_m}(x)\bigr)$ where $A_{i,m}:=A_{a_i,m}$ is the $m$-th Taylor coefficient provided by
Lemma~\ref{lem:taylor-bound}.
By multilinearity, the decomposition \eqref{eq:x_ai_patch-blocks_exponential}, and the
multilinear Taylor estimate, we obtain the expansion
\begin{equation}\label{eq:full-block-expansion_exp}
u(x)
=
u(a_i)
+
\sum_{m=1}^{\infty}
\sum_{\mathbf r\in\N^m}
F_{i,m,\mathbf r}(x),
\end{equation}
which converges absolutely in $Y$, uniformly with respect to
$x\in\mathcal{K}$.
We now group the ordered multi-indices according to their
multiplicities. For
$
\mathbf r=(r_1,\ldots,r_m)\in\mathbb N^m,
$
define its multiplicity sequence
$\nu=(\nu_q)_{\{q\geq0\}}$ by
$$
\nu_q
:=
\#\{j\in\{1,\ldots,m\}:r_j=q\},
\qquad q\geq0.
$$
In fact, $\nu_q$ is the multiplicity of $g_{i,q}(x)$ in the ordered
family
$
\bigl(g_{i,r_j}(x)\bigr)_{\{j=1\}}^m.
$
Thus, $\nu=(\nu_q)_{\{q\geq0\}}\in\mathfrak N$ and $|\nu|=m$. Conversely, for any $\nu\in\mathfrak N$ with $|\nu|=m$, let
$$
\mathcal R_\nu
:=
\left\{
\mathbf r=(r_1,\ldots,r_m)\in\mathbb N^m:
\#\{j\in\{1,\ldots,m\}:r_j=q\}=\nu_q
\text{ for every }q\geq0
\right\}.
$$
Then we have the disjoint decomposition
\begin{equation}\label{Eq:N_disjointdecomposition}
\mathbb N^m
=
\bigcup_{\substack{\nu\in\mathfrak N\\|\nu|=m}}
\mathcal R_\nu.
\end{equation}
For any $\nu\in\mathfrak N$ with $|\nu|=m$, using $A_{i,m}$ is symmetric,
 we have
$$
F_{i,m,\mathbf r}(x)
=
A_{i,m}
\bigl(
g_{i,0}(x)^{\nu_0},
g_{i,1}(x)^{\nu_1},
\ldots
\bigr)\quad \text{for any }\mathbf r\in\mathcal R_\nu.
$$
Thus, 
\begin{equation}\label{eq:sum_r_F}
\sum_{\mathbf r\in\mathcal R_\nu}
F_{i,m,\mathbf r}(x)
=
\#\mathcal R_\nu\,
A_{i,m}
\bigl(
g_{i,0}(x)^{\nu_0},
g_{i,1}(x)^{\nu_1},
\ldots
\bigr).
\end{equation}
Combining \eqref{Eq:N_disjointdecomposition} and \eqref{eq:sum_r_F}
with the identity
$
\#\mathcal R_\nu=m!/\nu!
$
yields
\begin{equation}
\sum_{\mathbf r\in\mathbb N^m}
F_{i,m,\mathbf r}(x)
=
\sum_{\substack{\nu\in\mathfrak N\\|\nu|=m}}
\sum_{\mathbf r\in\mathcal R_\nu}
F_{i,m,\mathbf r}(x)
=
\sum_{\substack{\nu\in\mathfrak N\\|\nu|=m}}
\frac{m!}{\nu!}
A_{i,m}
\bigl(
g_{i,0}(x)^{\nu_0},
g_{i,1}(x)^{\nu_1},
\ldots
\bigr),\label{eq:sum_F_m_equality}
\end{equation}
where
$g_{i,r}(x)^{\nu_r}$ means that $g_{i,r}(x)$ occurs as an argument of
$A_{i,|\nu|}$ exactly $\nu_r$ times.
The regrouping is justified by the absolute convergence of
\eqref{eq:full-block-expansion_exp}.
For $i\in\{1,\ldots,M\}$ and
$\nu\in\mathfrak N\setminus\{0\}$, define
\begin{equation}
\label{eq:exponential-F-block}
F_{i,\nu}(x)
:=
\frac{|\nu|!}{\nu!}
A_{i,|\nu|}
\bigl(
g_{i,0}(x)^{\nu_0},
g_{i,1}(x)^{\nu_1},\ldots
\bigr).
\end{equation}
By \eqref{eq:sum_F_m_equality} and
\eqref{eq:exponential-F-block}, the expansion
\eqref{eq:full-block-expansion_exp} can be written as
\begin{equation}
\label{eq:full-multiplicity-expansion}
u(x)
=
u(a_i)
+
\sum_{\nu\in\mathfrak N\setminus\{0\}}
F_{i,\nu}(x).
\end{equation}

\medskip
\noindent\textbf{Step 4: Block amplitude estimate.} 
For each $\nu\in\mathfrak N$, define
$$
E(\nu)
:=
-|\nu|\log q
+
\frac{c_G}{2}
\sum_{r\geq0}\nu_r2^{\alpha(J+r)},
$$
where $q=\sum_{r=0}^\infty\beta_r\in(0,1)$ is fixed as in $\textbf{Step 3}$. By \eqref{eq:multilinear-cauchy},
\eqref{eq:exponential-block-size}, and
\eqref{eq:exponential-F-block}, we obtain
$$
\|F_{i,\nu}(x)\|_Y
\leq
C\frac{|\nu|!}{\nu!}
\prod_{r\geq0}
\left(
\frac{\ee C_G}{\rho}\ee^{-c_G2^{\alpha(J+r)}}
\right)^{\nu_r}.
$$
To estimate the norm of each block, we split the exponential decay
inside the parentheses on the right-hand side into two equal factors.
Since
$
\beta_r=(\ee C_G/\rho)
\exp\bigl(-\frac{c_G}{2}2^{\alpha(J+r)}\bigr),
$
we have
$$
\begin{aligned}
\frac{|\nu|!}{\nu!}
\prod_{r\geq0}
\left(
\frac{\ee C_G}{\rho}\ee^{-c_G2^{\alpha(J+r)}}
\right)^{\nu_r}
&=
\left(
\frac{|\nu|!}{\nu!}
\prod_{r\geq0}\beta_r^{\nu_r}
\right)
\exp\left(
-\frac{c_G}{2}
\sum_{r\geq0}\nu_r2^{\alpha(J+r)}
\right)
\end{aligned}
$$
By the multinomial formula,
$$
\frac{|\nu|!}{\nu!}
\prod_{r\geq0}\beta_r^{\nu_r}
\leq
\left(\sum_{r\geq0}\beta_r\right)^{|\nu|}
=
q^{|\nu|}.
$$
Thus, it holds that
\begin{equation}
\label{eq:exponential-energy-bound}
\|F_{i,\nu}(x)\|_Y
\leq
C\ee^{-E(\nu)}\quad \text{for any } \nu\in\mathfrak N.
\end{equation}

\medskip
\noindent\textbf{Step 5: Partition-function estimate.} 
For each $i\in{1,2,\ldots,M}$ and
$\nu\in\mathfrak N\setminus{0}$, with $|\nu|=m$, define
\begin{equation}
\label{eq:W-i-nu}
W_{i,\nu}
:=
\operatorname{span}
\left\{
A_{i,m}
\bigl(
v_{0,1},\ldots,v_{0,\nu_0},
v_{1,1},\ldots,v_{1,\nu_1},
\ldots
\bigr)
:
v_{r,\ell}\in V_{J+r}
\right\}
\subset Y,
\end{equation}
where $r\geq0$ and $1\leq\ell\leq\nu_r$, with the convention that, whenever $\nu_k=0$, the block of arguments corresponding to $r=k$ is omitted.
Since $\nu$ is finitely
supported, only finitely many arguments occur in the definition. For a finite-dimensional vector space $V$ and $m\in\mathbb N$, 
we denote by $\operatorname{Sym}^{m}(V)$ the $m$-th symmetric tensor power of $V$.
It is well known that
$
\dim\operatorname{Sym}^{m}(V)
=
\binom{\dim V+m-1}{m}
$
; see, e.g., \cite{FultonHarris1991}. Since $A_{i,|\nu|}$ is symmetric and $\dim V_{J+r}\leq d_r=2^{J+r+1}$, we have
\begin{equation}
\label{eq:exponential-block-dimension}
\dim W_{i,\nu}
\leq \prod_{r\geq0}
\dim\bigl(\operatorname{Sym}^{\nu_r}(V_{J+r})\bigr)\leq
\prod_{r\geq0}
\binom{d_r+\nu_r-1}{\nu_r}:=D_\nu
.
\end{equation}
Here the factors corresponding to $\nu_r=0$ are understood to be equal to $1$, and the product is finite because $\nu$ is finitely supported. Next, we verify that $E(\nu)$ defined in \textbf{Step 4}
satisfy condition \eqref{eq:partition-abstract} of
Lemma~\ref{Le:block-truncation_exponential_width}. Recall that
$$
E(\nu)
=
\delta|\nu|
+
\gamma\sum_{r\geq0}\nu_r2^{\alpha(J+r)},
$$
where $\delta=-\log q>0$ and $\gamma=c_G/2>0$. For any
$\tau\in(0,1]$, define
$$
Z(\tau)
:=
\sum_{\nu\in\mathfrak N}
D_\nu \ee^{-\tau E(\nu)}.
$$
It follows from \eqref{eq:exponential-block-dimension} that
\begin{align}
Z(\tau)
&=
\sum_{\nu\in\mathfrak N}
\prod_{r\geq0}
\left[
\binom{d_r+\nu_r-1}{\nu_r}
\ee^{-\tau\nu_r(\delta+\gamma2^{\alpha(J+r)})}
\right]\notag
\\
&=
\prod_{r=0}^{\infty}
\left[
\sum_{k=0}^{\infty}
\binom{d_r+k-1}{k}
\ee^{-\tau k(\delta+\gamma2^{\alpha(J+r)})}
\right]\notag
\\
&=
\prod_{r=0}^{\infty}
\left(
1-\ee^{-\tau(\delta+\gamma2^{\alpha(J+r)})}
\right)^{-d_r}\label{eq:exponential-partition-product},
\end{align}
where the last equality follows from the generating function
$$
\sum_{k=0}^{\infty}
\binom{d+k-1}{k}z^k
=
(1-z)^{-d},
\qquad z\in (-1,1).
$$
Taking logarithms of \eqref{eq:exponential-partition-product} and using
the power-series identity
$$
-\log(1-z)
=
\sum_{\ell=1}^{\infty}\frac{z^\ell}{\ell},
\qquad z\in (0,1),
$$
we obtain
\begin{equation}\label{eq:partitionlog}
\log Z(\tau)
=
\sum_{r=0}^{\infty}
d_r
\sum_{\ell=1}^{\infty}
\frac{
\ee^{-\ell\tau(\delta+\gamma2^{\alpha(J+r)})}
}{\ell}
=
\sum_{\ell=1}^{\infty}
\frac{\ee^{-\ell\tau\delta}}{\ell}
\sum_{r=0}^{\infty}
d_r
\ee^{-\ell\tau\gamma2^{\alpha(J+r)}}.
\end{equation}
A standard dyadic sum--integral comparison shows that there exists a
constant $C_\alpha>0$, independent of $b$ and $J$, such that
\begin{equation}
\label{eq:dyadic-exponential-sum}
\sum_{r=0}^{\infty}
2^{J+r}\ee^{-b2^{\alpha(J+r)}}
\leq
C_\alpha b^{-1/\alpha},
\qquad \text{for any }b>0.
\end{equation}
Since $d_r=2^{J+r+1}$, applying
\eqref{eq:dyadic-exponential-sum} with $b=\ell\tau\gamma$ gives
$$
\sum_{r=0}^{\infty}
d_r \ee^{-\ell\tau\gamma2^{\alpha(J+r)}}
\leq
C_{\gamma}(\ell\tau)^{-1/\alpha}\quad  \text{for any } \ell \in \mathbb N \text{ and any } \tau\in (0,1],
$$
where $C_{\gamma}=2\gamma^{-1/\alpha} C_\alpha$ is independent of $\ell$ and $\tau$. Substituting this
estimate into \eqref{eq:partitionlog} yields
\begin{align}
\log Z(\tau)
&\leq
C_{\gamma}\tau^{-1/\alpha}
\sum_{\ell=1}^{\infty}
\ell^{-1-1/\alpha}\ee^{-\ell\tau\delta}
\leq
C_{\gamma}\tau^{-1/\alpha}
\sum_{\ell=1}^{\infty}
\ell^{-1-1/\alpha}
\leq
C_1\tau^{-1/\alpha}\notag,
\end{align}
where $C_1$ is independent of $\tau$.
Therefore,
\begin{equation}
\label{eq:partition-final}
Z(\tau)
\leq
\exp\left(C_1\tau^{-1/\alpha}\right),
\quad \text{for any } \tau\in (0,1].
\end{equation}
Finally, since the full family of blocks is indexed by
$(i,\nu)\in\{1,\ldots,M\}\times
(\mathfrak N\setminus\{0\})$, we have
\begin{equation}\label{Eq:sum_D_E_bound}
\sum_{i=1}^{M}
\sum_{\nu\in\mathfrak N\setminus\{0\}}
D_\nu \ee^{-\tau E(\nu)}
\leq
MZ(\tau)
\leq
\exp\left(C_2\tau^{-1/\alpha}\right),
\quad \text{for any } \tau\in (0,1],
\end{equation}
after enlarging the constant $C_2$. Thus
condition \eqref{eq:partition-abstract} is satisfied.

\vspace{2mm}
\noindent\textbf{Step 6: Upper bounds of Kolmogorov $n$-width.}
We now apply the abstract Lemma \ref{Le:block-truncation_exponential_width} globally, exactly as in the algebraic proof. Define
$$
Y_0:=\spann\{u(a_1),\ldots,u(a_M)\},
\qquad
\dim Y_0\leq M,
$$
and index the blocks by $(i,\nu)$. For a fixed $x\in\mathcal K$, only the blocks corresponding to the selected patch $i(x)$ occur in \eqref{eq:full-multiplicity-expansion}; all blocks belonging to the other patches are set equal to zero. 
Combining \eqref{eq:exponential-energy-bound} and
\eqref{Eq:sum_D_E_bound}, and applying
Lemma~\ref{Le:block-truncation_exponential_width} with $d_0=M$, we obtain
\begin{equation}
\label{eq:global-exponential-rate-minus-M}
d_n\bigl(u(\mathcal K)\bigr)_Y
\leq
C\exp\left(
-c\bigl[\log(n-M+1)\bigr]^{\alpha+1}
\right),
\qquad n>M.
\end{equation}
For $n\geq2M$, one has $n-M+1\geq(n+1)/2$, and therefore the right-hand side of \eqref{eq:global-exponential-rate-minus-M} is bounded by
$
C'\exp\left(-c'[\log(n+2)]^{\alpha+1}\right).
$
The finitely many indices $1\leq n<2M$ are absorbed by increasing the constant. 
This completes the proof.

\begin{remark}
In the algebraic decay case, the Taylor expansion is indexed by ordered
multi-indices
$
\mathbf r=(r_1,\ldots,r_m)\in\mathbb N^m,
$
so that each permutation is treated as a separate block. This is
sufficient because the corresponding weighted sum factorizes into a
geometric series. In the exponential case, we instead group together
all ordered multi-indices having the same multiplicities. 
The multiplicity formulation allows us to exploit symmetric tensor
dimensions and derive the partition-function estimate required for the
sharp exponential bound. Thus, the two definitions represent the same
Taylor expansion, but the latter provides the finer counting needed in
the exponential regime.

\end{remark}
\subsection{Proof of Theorem \ref{thm:main_result4}}
\begin{proof}
Let $X:=\ell_2(\N_+;\C)$
and define the set
\begin{equation}\label{eq:boxK}
\mathcal{K}:=\left\{x=(x_j)_{j\ge1}\in\ell_2:
 |x_j|\le a_j,\quad
 a_j:=\ee^{-\kappa j^\alpha},\ \text{for every }j\right\},
\end{equation}
where $\kappa>0$ is a  fixed constant. Compactness follows because the coordinate tails are uniformly small in $\ell_2$. Let $\mathfrak S$ be the collection of all finite nonempty subsets of $\N_+$, and set
$$
Y:=\ell_2(\mathfrak S;\C).
$$
Fix $0<\eta<1$ and define
\begin{equation}\label{eq:entire-map}
u(x)_S:=\eta^{|S|}\prod_{j\in S}x_j,
 \qquad \text{for any }S\in\mathfrak S.
\end{equation}
By following the same argument as in the proof of
Theorem~\ref{thm:main_result2}, we can verify that
$\mathcal K$ and $u$ satisfy the corresponding assumptions
of Theorem~\ref{thm:main_result4} and that of Theorem \ref{thm:main_result3}
is also satisfied. Indeed, for any $x\in X$, we have
\begin{align*}
 \|u(x)\|_Y^2
 =\sum_{S\in\mathfrak S}
 \eta^{2|S|}\prod_{j\in S}|x_j|^2
 =\prod_{j=1}^\infty(1+\eta^2|x_j|^2)-1\le \ee^{\eta^2\|x\|_X^2}-1.
\end{align*}
Hence $u$ is holomorphic and bounded on bounded subsets of $X$.
Then the upper bound of Kolmogorov $n$-width of $u(\mathcal{K})$ follows from Theorem \ref{thm:main_result3}.
It remains to prove the lower bound. Given any integer $N\ge2$, set
$$
 m_N:=\lfloor N/2\rfloor\quad
 \text{and}\quad
 \mathfrak S_N:=\{S\subset\{1,\ldots,N\}:|S|=m_N\}.
$$
Consider the subset $\mathcal S_N\subset\mathcal K$ defined by
$$
\mathcal S_N
:=
\left\{
x^S:=\sum_{j\in S}a_j e_j\in \mathcal{K}
:
S\in\mathfrak S_N
\right\}.
$$
Then $|\mathcal{S}_N|=|\mathfrak{S}_N|=\binom{N}{m_N}:=M_N
 \ge
\frac{2^N}{N+1}$. Let $(f_S)_{S\in\mathfrak S}$ be an orthonormal basis of $Y$ and $P_N:Y\to Y_N$ denote the orthogonal projection onto
$
Y_N
:=
\operatorname{span}
\left\{
f_S:S\in\mathfrak S_N
\right\}
\subset Y$. Consequently,
\begin{equation}\label{eq:projected-spike}
 P_Nu(x^S)
 =\eta^{m_N}\left(\prod_{j\in S}a_j\right)f_S\quad \text{for any }x^S\in \mathcal{S}_N.
\end{equation}
These vectors are mutually orthogonal. Moreover, by the definitions
$m_N:=\lfloor N/2\rfloor$ and $a_j:=\ee^{-\kappa j^\alpha}$, we have
\begin{equation}
\label{eq:spike-amplitude}
\eta^{m_N}\prod_{j\in S}a_j
=
\exp\left(
m_N\log\eta-\kappa\sum_{j\in S}j^\alpha
\right)
\geq
\exp\left(-CN^{\alpha+1}\right),\quad  \text{for any }
S\in\mathfrak S_N,
\end{equation}
where $C>0$ is independent of $N$ and $S$.
Let $n_N:=\lfloor M_N/2\rfloor$. Since the Kolmogorov widths do not increase under linear contractions and
$P_N$ is an orthogonal projection with $\|P_N\|=1$, we have
$$
d_{n_N}\bigl(P_Nu(\mathcal K)\bigr)_Y
\leq d_{n_N}\bigl(u(\mathcal K)\bigr)_Y
.
$$
Then, by $P_Nu(\mathcal{S}_N)\subset P_Nu(\mathcal{K})$ in \eqref{eq:projected-spike} and using Lemma \ref{lem:orthogonal-spikes} and \eqref{eq:spike-amplitude}, we obtain
\begin{equation}\label{eq:subsequence-lower}
 d_{n_N}(u(\mathcal{K}))_Y
 \ge d_{n_N}(P_Nu(\mathcal{S}_N))_Y
 \ge c\exp(-CN^{\alpha+1}).
\end{equation}
Since $\log n_N\asymp N$, this is
\[
 d_{n_N}(u(\mathcal{K}))_Y
 \ge c\exp\bigl(-C[\log(n_N+2)]^{\alpha+1}\bigr).
\]
The sequence $n_N$ grows geometrically up to polynomial factors. Monotonicity of $d_n$ therefore extends the estimate, after changing constants, to every sufficiently large $n$. The finitely many remaining values are absorbed into the constants. The proof is complete.
\end{proof}

\section{Conclusion}\label{Sec:conclusion}

We have established sharp general estimates for Kolmogorov
$n$-widths of compact sets under holomorphic mappings between complex
Banach spaces. In the algebraic regime, we establish
$$
d_n(\mathcal K)_X=O(n^{-s})
\quad
\text{for some } s>0 \Longrightarrow
d_n(u(\mathcal K))_Y=O(n^{-t})
\quad 
\text{for every } t<s.
$$
This improves the previously known loss of one full
power and is optimal in general, since the endpoint $t=s$ may fail even
for continuous two-homogeneous polynomial mappings.
In the exponential regime, we prove 
$$
d_n(\mathcal K)_X
\leq
C_0\exp(-c_0n^\alpha)
\quad\Longrightarrow\quad
d_n(u(\mathcal K))_Y
\leq
  C_1\exp\left(-c_1[\log(n+2)]^{\alpha+1}\right).
$$
A matching construction shows that this rate is sharp for general holomorphic mappings. These results clarify
the  difference between algebraic and exponential width
decay and identify the optimal universal behavior in both regimes.
Determining additional structural assumptions that ensure endpoint
preservation in the algebraic regime remains to be a direction for future research.

\section*{Funding}
This work was supported by the National Natural Science Foundation of China under grant
12471346.

\section*{Declaration of generative AI use}
During the preparation of this work, the authors used ChatGPT to assist with language refinement, manuscript organization,
and the presentation of mathematical arguments. After using this tool,
the authors reviewed and edited the content as needed, independently
verified all mathematical results, and take full responsibility for the
content of the publication.

\bibliographystyle{elsarticle-num}

\end{document}